\documentclass[11 pt]{article}

\usepackage{amsfonts,amscd,amssymb,amsthm}
\usepackage{amsmath, enumerate}
\usepackage{amsxtra}
\usepackage{amscd}
\usepackage{eucal}
\usepackage{array, caption}
\usepackage{graphicx, multicol, pdfpages}
\usepackage{tikz}
\usepackage{url}

\newtheorem{theorem}{Theorem}[section]
\newtheorem{lemma}[theorem]{Lemma}
\newtheorem{proposition}[theorem]{Proposition}
\newtheorem{corollary}[theorem]{Corollary}
\newtheorem{definition}[theorem]{Definition}

\theoremstyle{definition}
\newtheorem{remark}[theorem]{Remark}

\newtheorem{question}[theorem]{Question}

\DeclareMathOperator{\reg}{reg}

\DeclareMathOperator{\st}{st}
\DeclareMathOperator{\lex}{lex}
\DeclareMathOperator{\pol}{pol}

\title{Powers of Edge Ideals with Linear Resolutions}
\author{Nursel Erey}
 \author{Nursel Erey\footnote{Department of Mathematics, North Dakota State University, Fargo, ND   \textit{e-mail}: nurselerey@gmail.com \qquad \qquad \textit{2010 Mathematics Subject Classification.} 13D02, 05E40   \quad  \textit{Keywords}: Edge Ideal, Castelnuovo-Mumford Regularity} }

\begin{document}

\date{} 
\maketitle


\begin{abstract}
We show that if $G$ is a gap-free and diamond-free graph, then $I(G)^s$ has a linear minimal free resolution for every $s\geq 2$. 
\end{abstract}


\section{Introduction}

Let $S=k[x_1,\dots ,x_n]$ be a polynomial ring and let $G$ be a finite simple graph with edge ideal $I(G)$. There has been recent interest in studying the asymptotic regularity of powers of edge ideals for various families of graphs \cite{banerjee, selvi, herzog hibi zheng, moghimian et al, peeva nevo}. It is known \cite{kodiyalam} that when $I$ is a graded ideal, regularity of powers $I^s$ is a linear function for large $s$. In particular, for any graph $G$, there exist constants $c$ and $s_0$ such that $\reg(I(G)^s)=2s+c$ for all $s\geq s_0$. A special case of interest is to combinatorially characterize when the powers of edge ideals have eventually linear minimal free resolutions. Since the $s$th power of the edge ideal is generated in degree $2s$, the ideal $I(G)^s$ has a linear minimal free resolution if and only if $c=0$ in the given formula.

It was proved by Francisco, H\`{a} and Van Tuyl that if some power of $I(G)$ has linear minimal free resolution, then the complement graph $G^c$ has no induced $4$-cycles, i.e., $G$ is gap-free. In this direction, Nevo and Peeva \cite{peeva nevo} raised the following question.

\begin{question}
Is it true that a graph $G$ is gap-free if and only if $I(G)^s$ has a linear minimal free resolution for every $s\gg 0$?
\end{question}

This question is known to have a positive answer when the regularity of the edge ideal is $2$. In fact, Herzog, Hibi and Zheng \cite{herzog hibi zheng} proved that if $I(G)$ has a linear minimal free resolution, then so does $I(G)^s$ for all $s\geq 1$. Banerjee \cite{banerjee} proved that if $G$ is a gap-free and cricket-free graph, then $I(G)^s$ has linear minimal free resolution for every $s\geq 2$. In this paper, we provide a new family of gap-free graphs for which the second and higher powers of edge ideals have linear minimal free resolutions.

\section{Preliminaries}
\subsection{Graph theory}

Throughout this paper we consider finite simple graphs, i.e., graphs with no loops or multiple edges. For any graph $G$ we write $V(G)$ and $E(G)$ respectively for the vertex and edge sets of the graph. We say two vertices $u$ and $v$ are \textbf{adjacent}  and write $uv\in G$, if there is an edge between them. The set $N(u)=\{v: v \text{ is adjacent to } u \}$ is called the \textbf{neighbor set} of $u$. Any $v\in N(u)$ is called a \textbf{neighbor} of $u$. A vertex is called an \textbf{isolated vertex} if it has no neighbors. A \textbf{complete graph} (or \textbf{clique}) is a graph such that every pair of vertices are adjacent. A complete graph on $n$ vertices is denoted by $K_n$.

A graph $H$ is called a \textbf{subgraph} of $G$ if the vertex and edge sets of $H$ are contained respectively in those of $G$. A subgraph $H$ of $G$ is called an \textbf{induced subgraph} if $uv\in G$ implies $uv\in H$ for every $u,v\in V(H)$. A subset $A$ of the vertices of $G$ is called \textbf{independent} if $uv\notin G$ for every $u,v\in A$. A \textbf{dominating clique} of a graph $G$ is a complete subgraph $H$ of $G$ such that every vertex of $G$ either belongs to $H$ or is adjacent to some vertex of $H$.

 For any graph $K$, we say $G$ is \textbf{$K$-free} if $K$ is not an induced subgraph of $G$. The \textbf{clique number} of $G$, denoted $\omega(G)$, is the number of vertices in a maximum clique of $G$. For any vertex $v$ of $G$, the graph $G-v$ denotes the induced subgraph of $G$ which is obtained from $G$ by \textbf{removing the vertex} $v$. The \textbf{complement} of a graph $G$, denoted $G^c$, is a graph on the same vertex set such that for every vertex $u$ and $v$, $uv\in G^c$ if $uv\notin G$.

Let $F$ and $G$ be graphs with no common vertices and let $v\in V(G)$. We say $G'$ is obtained from $G$ by \textbf{substituting} $F$ for $v$ if the vertex $v$ is replaced by $F$ and all the vertices of $N(v)$ are adjacent to all the vertices of $F$, i.e., $V(G')=(V(G)\cup V(F))\setminus \{v\}$ and $E(G')=E(G-v)\cup E(F)\cup \{ab: a\in N(v), b\in V(F)\}$. For a vertex $v\in V(G)$ and a positive integer $k$, we say $G'$ is obtained from $G$ by \textbf{multiplying} $v$ by $k$ when $G'$ is formed by substituting an independent set $A$ of $k$ vertices for $v$. 

A \textbf{cycle} graph with vertices $v_1,\dots ,v_n$ and edges $v_1v_2, \dots , v_{n-1}v_n, v_nv_1$ is denoted by $C_n=(v_1v_2\dots v_n)$. The complement of a cycle graph is called an \textbf{anticycle}. A \textbf{path} in a graph $G$ is a sequence of vertices $u_1,\dots, u_n$ such that $u_iu_{i+1}\in G$ for all $1\leq i\leq n-1$. A \textbf{path graph} on $n$ vertices, denoted by $P_n$, is a graph with $n$ distinct vertices $u_1,\dots , u_n$ and $n-1$ edges $u_1u_2, u_2u_3, \dots , u_{n-1}u_n$. A graph which is isomorphic to the graph with vertices $a, b, c, d$ and edges $ab, bc, ac, ad, cd$ is called a \textbf{diamond}. A graph which is isomorphic to the graph with vertices $w_1, w_2, w_3, w_4, w_5$ and edges $w_1w_3, w_2w_3, w_3w_4, w_3w_5, w_4w_5$ is called a \textbf{cricket}. 

 A graph $G$ is called \textbf{chordal} if every induced cycle in $G$ has at most three vertices. A \textbf{bipartite} graph is a graph that does not contain any odd-length cycles. Two disjoint edges $uv$ and $ab$ of $G$ are said to form a \textbf{gap} if the induced subgraph of $G$ on the vertices $u, v, a, b$ has only two edges. A graph $G$ is called \textbf{gap-free} if no two edges form a gap, or equivalently, $G^c$ is $C_4$-free.
 
For any $v\in V(G)$ and any induced subgraph $H$ of $G$, the \textbf{distance} between $v$ and $H$ is the number of edges in a shortest path from $v$ to a vertex of $H$. In particular, the distance between $v$ and $H$ is zero when $v\in H$.

\subsection{Castelnuovo-Mumford regularity of monomial ideals}
Let $S=k[x_1,\dots ,x_n]$ be a polynomial ring over a field $k$. The \textbf{Castelnuovo-Mumford regularity} of a monomial ideal $I$ is given by
$$\reg(I)=\max\{j-i : b_{i,j}(I)\neq 0\}   $$
where $b_{i,j}(I)$ are the graded Betti numbers of $I$. If $G$ is a graph on the vertices $x_1,\dots , x_n$, then the \textbf{edge ideal} of $G$ is defined as
$$I(G)= (xy: xy \text{ is an edge of } G). $$
The method of polarization reduces the study of minimal free resolutions of monomial ideals to that of square-free monomial ideals. Therefore quadratic monomial ideals can be studied via edge ideals. We recall the following well-known results.

\begin{theorem}\cite[Corollary 1.6.3]{herzog hibi monomial ideals}\label{thm: polarization}
	Let $I\subseteq S$ be a monomial ideal and let $I^{\pol}$ be its polarization. Then $\reg(I)=\reg(I^{\pol})$.
\end{theorem}
\begin{lemma}\label{lem:regularity after adding variables} 
	If $I\subseteq S$ is a monomial ideal, then $\reg(I,x)\leq \reg(I)$ for any variable $x$.
\end{lemma}

We will make use of the following theorem to bound the regularity of monomial ideals.
\begin{theorem} \cite[Lemma~2.10]{huneke}\cite[Lemma~2.11]{banerjee}\label{thm:regularity bound by colon ideals}
	Let $I\subseteq S$ be a monomial ideal, and let $m$ be a monomial of degree $d$. Then
	$$\reg(I)\leq \max\{\reg(I:m)+d, \reg(I,m)\}. $$
	Moreover, if $m$ is a variable $x$ appearing in $I$, then $\reg(I)$ is equal to one of these terms. 
\end{theorem}

The next result is a restatement of Theorem~\ref{thm:regularity bound by colon ideals} in terms of edge ideals. Note that for any vertex $x$ of $G$, the set $\st x$ stands for $N(x)\cup\{x\}$.

\begin{theorem}\cite[Lemma~3.1]{huneke}\label{lem: regularity bound with star x} Let $x$ be a vertex of $G$ with neighbors $y_1,\dots ,y_m$. Then
	$$\reg(I(G))\leq \max\{\reg (I(G- \st x))+1, \reg(I(G- x)) \}. $$
	Moreover, $\reg(I(G))$ is equal to one of these terms.
\end{theorem}

Banerjee \cite{banerjee} introduced the following total order on the minimal monomial generators of powers of edge ideals. 

\begin{definition}\cite[Discussion~4.1]{banerjee}\label{def:order on the generators of the powers of ideal}
		Let $I$ be an edge ideal which is minimally generated by the monomials $L_1,\dots ,L_k$. Consider the order $L_1>L_2 >\cdots > L_k$ and let $n\geq 2$. For any minimal monomial generator $M$ of $I^n$, the expression $M=L_1^{a_1}L_2^{a_2}\dots L_k^{a_k}$ is called \textbf{maximal} if for all expressions $M=L_1^{b_2}\dots L_k^{b_k}$, we have $(a_1,\dots ,a_k)>_{\lex} (b_1,\dots ,b_k)$. Given two minimal monomial generators $M, N$ of $I^n$ with maximal expressions $M=L_1^{a_1}L_2^{a_2}\dots L_k^{a_k}$ and $N=L_1^{c_1}L_2^{c_2}\dots L_k^{c_k}$ we set $M>N$ if $(a_1,\dots ,a_k)>_{\lex} (c_1,\dots ,c_k)$. We let $L^{(n)}$ denote the totally ordered set of minimal monomial generators of $I^n$.
\end{definition}

The following theorem follows from \cite[Theorem~4.12]{banerjee}.

\begin{theorem}\cite{banerjee}\label{thm:ordered colon ideals}
	Let $I= (L_1,\dots ,L_k)$ be an edge ideal and let $s\geq 1$. Let $L^{(s)}: L_1^{(s)}>L_2^{(s)}>\cdots >L_r^{(s)}$ be the order on the minimal monomial generators of $I^s$ induced by the order $L_1>\cdots >L_k$ as described in Definition~\ref{def:order on the generators of the powers of ideal}. Then for every $1\leq \ell \leq r-1$, $$((I^{s+1}, L_1^{(s)},\dots , L_{\ell}^{(s)}):L_{\ell+1}^{(s)})= ( (I^{s+1}:L_{\ell+1}^{(s)}), \text{ some variables}).$$
\end{theorem}

As a consequence of the theorem above, the following result was obtained in \cite{banerjee}, which gives a sufficient condition for powers of edge ideals to have linear resolutions.

\begin{theorem}\cite[Corollary~5.3]{banerjee}\label{thm: banerjee main bound} Let $G$ be a graph with edge ideal $I=I(G)$. If $\reg(I)\leq 4$ and for all $s\geq 1$ and for each minimal monomial generator $m$ of $I^s$, $\reg(I^{s+1}:m)\leq 2$, then $I^t$ has linear minimal free resolution for every $t\geq 2$.
\end{theorem}

The following classical result of Fr\"oberg \cite{froberg} gives a combinatorial characterization of edge ideals which have linear minimal free resolutions.
\begin{theorem}\cite[Theorem~1]{froberg}\label{thm: frobergs theorem} The minimal free resolution of $I(G)$ is linear if and only if the complement graph $G^c$ is chordal.
\end{theorem}
\subsection{Even-connection and gap-free graphs}
	In \cite[Definition~6.2]{banerjee} the concept of even-connectedness was defined in order to describe in combinatorial terms the minimal generators of certain colon ideals of powers of edge ideals.
\begin{definition}
	Let $G$ be a graph. Two vertices $u$ and $v$ in $G$ are said to be \textbf{even-connected} with respect to an $s$-fold product $e_1\dots e_s$ of edges of $G$ if there is a sequence of vertices $p_0,p_1,\dots , p_{2\ell+1}$ for some $\ell \geq 1$ in $G$ such that
	\begin{enumerate}[$(i)$]
		\item  $p_ip_{i+1}\in G$ for all $0\leq i \leq 2\ell$,
		\item $p_0=u$ and $p_{2\ell+1}=v$,
		\item  for all $0\leq j\leq \ell-1$, $p_{2j+1}p_{2j+2}=e_i$ for some $i$, and
		\item  for all $k$, $\lvert \{j \mid \{p_{2j+1}, p_{2j+2}\}=e_k\}  \rvert \leq \lvert \{t \mid e_t=e_k\} \rvert$.
		\end{enumerate}
		In such case, $p_0, p_1, \dots ,  p_{2\ell+1}$ is called an \textbf{even-connection} between $u$ and $v$ with respect to $e_1\dots e_s$.
\end{definition}

\begin{theorem}\cite[Theorems 6.1, 6.5 and 6.7]{banerjee}\label{thm:edges of associated graph to colon ideal}
	Let $G$ be a graph with edge ideal $I=I(G)$ and let $s\geq 1$. Let $M=e_1\dots e_s$ be a minimal monomial generator of $I^s$ where $e_1,\dots , e_s$ are some edges of $G$. Then $(I^{s+1}:M)$ is minimally generated by monomials of degree $2$, and $uv$ is a minimal generator of  $(I^{s+1}:M)$ if and only if either $uv\in G$ or $u$ and $v$ are even-connected with respect to $M$.
\end{theorem}

Consequently, if $G$ is a graph and $M$ is a minimal generator of $I(G)^s$ for some $s\geq 1$, then $(I(G)^{s+1}:M)^{\pol}$ is the edge ideal of some associated graph $G'$. When $G$ is a gap-free graph, some combinatorial properties of $G'$ are summarized in the next result.

\begin{theorem}\cite[Lemmas 6.14 and 6.15]{banerjee}\label{thm: associated graph to colon ideal}
		Let $G$ be a gap-free graph with edge ideal $I=I(G)$ and let $e_1,\dots ,e_s$ be some edges of $G$ where $s\geq 1$. Then the graph $G'$ which is associated to $(I^{s+1}:e_1\dots e_s)^{\pol}$ is gap-free as well. Also if $C_n=(u_1\dots u_n)$ is a cycle on $n\geq 5$ vertices such that $C_n^c$ is an induced subgraph of $G'$, then $C_n^c$ is an induced subgraph of $G$ as well.
\end{theorem}

The next theorem is one of the main graph theoretic tools that will be used in the sequel.
\begin{theorem}\cite[Theorems 2 and 3]{chung}\label{thm: gap free and clique number at least 3}
	Let  $G$ be a gap-free graph with no isolated vertices. 
	\begin{enumerate}[$(i)$]
	\item If $\omega(G)\geq 3$, then $G$ has a dominating clique on $\omega(G)$ vertices.
         \item If $w(G)=2$ and  $G$ is not bipartite, then $G$ can be obtained  from  a five-cycle  by vertex multiplication. 
	\end{enumerate}
	\end{theorem}


\section{Colon ideals of powers of edge ideals}

In this section, we first prove some graph theoretic results regarding the structure of (gap, diamond)-free graphs. Then we analyze the colon ideals of the form $(I(G)^{s+1}:m)$ where $m$ is a minimal monomial generator of $I(G)^s$.
\begin{proposition}\label{prop: multiplying vertices}
	Let $H$ be a graph which is obtained from $G$ by multiplying a vertex $v\in V(G)$ by some $k\geq 2$.
	\begin{enumerate}[$(i)$]
		\item $G$ is gap-free if and only if $H$ is gap-free.
		\item $G$ and $H$ have the same clique number.
		\item If $G$ is diamond-free, then $H$ is diamond-free if and only if $v$ does not belong to any triangle of $G$.
	\end{enumerate}
\end{proposition}
\begin{proof} Suppose that $v$ is replaced with the independent set $\{v_1,\dots ,v_k\}$.
	
	$(i)$ If $H$ is gap-free, then $G$ is gap-free since $G$ can be considered as an induced subgraph of $H$. Assume for a contradiction $G$ is gap-free but there is a gap between $ab, cd\in H$. Observe that since the vertices $v_1,\dots , v_k$ have the same neighbors in $H$, at most one of them belongs to $\{a, b, c, d\}$. Then this gives a gap in $G$.
	
	$(ii)$ Since every clique of $H$ contains at most one vertex of $\{v_1, \dots ,v_k\}$, the result follows.
	
	$(iii)$ Suppose that $G$ is diamond-free. If $v$ belongs to a triangle of $G$, then $H$ clearly has an induced diamond. Conversely, suppose that $H$ has an induced diamond with edges $ab, ad, bd, bc, cd$. Since $G$ is diamond-free, without loss of generality we may assume that either $v_1=a$ or $v_1=d$. If $v_1=a$, then $(vbd)$ is a triangle in $G$. If $v_1=d$, then $(avb)$ is a triangle in $G$.
\end{proof}
\begin{lemma}\label{lem: diamond free structure}
	Let $G$ be a gap-free and diamond-free graph with $\omega(G)\geq 3$ and with no isolated vertices. Then there is a dominating clique $K_{\omega(G)}$ such that every vertex outside of $K_{\omega(G)}$ is adjacent to exactly one vertex of $K_{\omega(G)}$. Moreover the following statements hold.
	\begin{enumerate}[$(i)$] 
		\item If $\omega(G)\geq 4$, then $G-K_{\omega(G)}$ is independent and in particular, $G^c$ is chordal.
		\item If $\omega(G)=3$, then for every vertex $x$ of $K_3$ the set $N(x)\setminus V(K_3)$ is independent and in particular, $G-\st x$ is bipartite.
	\end{enumerate}
\end{lemma}
\begin{proof}
	Let $K_{\omega(G)}$ be a dominating clique on the vertices $u_1,\dots , u_{\omega(G)}$ as in Theorem~\ref{thm: gap free and clique number at least 3}.
	
	Suppose $v$ is a vertex of $G$ which is not in $K_{\omega(G)}$. Assume for a contradiction $vu_i\in G$ and $vu_j\in G$ for some $i\neq j$. Let $k$ be different from $i$ and $j$. Since $G$ is diamond-free, $u_kv \in G$. Since $k$ was arbitrary, it follows that $v$ is connected to all vertices of $K_{\omega(G)}$. Thus $G$ contains a clique with $\omega(G)+1$ vertices, a contradiction.
	
	(i) Let $v$ and $w$ be distinct vertices outside of $K_{\omega(G)}$. Assume for a contradiction $vw\in G$. Suppose $vu_i\in G$ and $wu_j\in G$ for some $i,j$. Since $\omega(G)\geq 4$, there exist distinct vertices $u_k$ and $u_{\ell}$ such that $\ell\neq i,j$ and $k\neq i,j$. We already know that none of $vu_k, vu_{\ell}, wu_k, wu_{\ell}$ is in $G$. Hence there is a gap between $vw$ and $u_ku_{\ell}$, a contradiction.
	
	(ii) Let $x ,y$ and $z$ be the vertices of $K_3$. Assume for a contradiction $ux, vx, uv\in G$ for some $u,v\notin V(K_3)$. Since $ux \in G$, $u$ cannot be adjacent to $y$ or $z$. Similarly $v$ cannot be adjacent to $y$ or $z$. Hence there is a gap between $uv$ and $yz$ which is a contradiction.
\end{proof}

\begin{remark}
	Observe that the complement of a $C_5$ is again a cycle on $5$ vertices. Therefore a graph is $C_5$-free if and only if it is $C_5^c$-free. We shall use this fact for the rest of the paper without reference.
\end{remark}
\begin{lemma}\label{lem: gap free and chordal families}
Suppose $G$ is a gap-free graph. 
\begin{enumerate}[$(i)$]
\item If $G$ is bipartite, then $G^c$ is chordal.
\item Let $G$ be connected ($C_5$, diamond)-free and $\omega(G)=3$. Then either $G=C_6^c$ or  $G^c$ is chordal.
\end{enumerate}
\end{lemma}

\begin{proof}
(i) Assume for a contradiction $G$ is bipartite but $G^c$ has an induced cycle $C_n=(x_1x_2\dots x_n)$ where $n\geq 4$. Since $G$ is gap-free, we must have $n\geq 5$. Let $A\cup B$ be a bipartition  of $G$ where $x_2\in A$. Then $x_4, x_5\in B$ since $x_2x_4, x_2x_5\in G$. As $x_1x_4\in G$ we get $x_1\in A$. Similarly, as $x_1x_3\in G$, we get $x_3\in B$. But then since $x_3x_5$ is an edge of $G$, the set $B$ is not independent, which is a contradiction.

(ii) Assume for a contradiction $G\neq C_6^c$ and $G^c$ is not chordal. Then $G^c$ has an induced $C_n=(x_1 \dots x_n)$ for some $n\geq 6$.

\emph{Case 1:} Let $n=6$. Then since $G\neq C_6^c$ and $G$ is connected, without loss of generality assume that there is a vertex $y\in V(G)\setminus V(C_6^c)$ such that $yx_1$ is an edge of $G$. Since there is no gap between $yx_1$ and $x_2x_6$, either $yx_6$ or $yx_2$ must be an edge of $G$. Without loss of generality assume that $yx_2\in G$. Since the induced subgraph on the vertices $x_1, x_3, x_6, x_2, y$ is not an induced $5$-cycle of $G$, we must have either $x_3y\in G$ or $x_6y\in G$. Without loss of generality suppose that $x_3y\in G$. Consider the induced subgraph $H$ on the vertices $x_1, x_3, x_5, y$. If $yx_5\notin G$, then $H$ is an induced diamond, contradiction. If $yx_5\in G$, then $H$ is a clique on $4$ vertices, which is again a contradiction as $\omega(G)=3$.

\emph{Case 2:} Let $n\geq 7$. Then the induced subgraph on the vertices $x_1,x_3,x_4, x_6$ is a diamond, contradiction.
\end{proof}


\begin{theorem}\label{thm: regularity of gap and diamond free}
Suppose $G$ is both gap-free and diamond-free. Then $\reg(I(G))\leq 3$.
\end{theorem}
\begin{proof}
If $G$ is triangle-free, then it is cricket-free and thus the result follows from \cite[Theorem~3.4]{banerjee}. If $\omega(G)\geq 4$, then the result follows from Lemma~\ref{lem: diamond free structure} and Theorem~\ref{thm: frobergs theorem}. Therefore we may assume that $\omega(G)= 3$. We proceed by induction on the number of vertices of $G$. If $G$ is a triangle, then $\reg(I(G))=2$. Let $K_3$ be a dominating clique as in Lemma~\ref{lem: diamond free structure} and let $x$ be a vertex of $K_3$. Since $G-x$ is either triangle-free or have clique number $3$ we get $\reg(I(G-x))\leq 3$ by induction. From Lemma~\ref{lem: diamond free structure} and Lemma~\ref{lem: gap free and chordal families} the graph $(G-\st x)^c$ is chordal and $\reg(I(G-\st x))=2$ follows from Theorem~\ref{thm: frobergs theorem}. Thus $\reg(I(G))\leq 3$ using Theorem~\ref{lem: regularity bound with star x}.
\end{proof}

\begin{lemma}\label{lem: anticycle does not intersect s fold product}
Let $G$ be a gap-free graph and let $e_1 \dots e_s$ be an $s$-fold product of edges. Let $C_n=(u_1\dots u_n)$ be a cycle where $n\geq 5$. Suppose $C_n^c$ is an induced anticycle in the graph which is associated to $(I(G)^{s+1} : e_1\dots e_s)^{\operatorname{pol}}$. Then $e_i\cap \{u_1,\dots , u_n\}=\emptyset$ for every $1 \leq i \leq s$.
\end{lemma}
\begin{proof}
Let $G'$ be the gap-free graph associated to $(I(G)^{s+1} : e_1\dots e_s)^{\operatorname{pol}}$. First note that by Theorem~\ref{thm: associated graph to colon ideal}, $C_n^c$ is also an induced anticycle in $G$. Let $e_i=xy$ be fixed. Notice that by symmetry, it is enough to consider the following two cases.

\emph{Case 1}: Suppose $\{x,y\}\cap \{u_1,\dots , u_n\}=\{x\}$. Without loss of generality assume that $x=u_1$. Since there is no gap between $xy$ and $u_2u_n$, either $yu_2\in G$ or $yu_n\in G$. Without loss of generality suppose $yu_2\in G$. Then $u_2yxu_3$ is an even-connection. Thus $u_2u_3\in G'$ which is a contradiction.

\emph{Case 2}: Suppose $\{x,y\}\cap \{u_1,\dots , u_n\}=\{x, y\}$. Without loss of generality, assume that $x=u_1$. Then $y=u_i$ for some $i\neq 2,n$. Since $n\geq 5$, either $i+1<n$ or $i>3$. If $i+1<n$, then $u_nyxu_{n-1}$ is an even-connection and $u_nu_{n-1}\in G'$, a contradiction. If $i>3$, then $u_2yxu_3$ is an even-connection and $u_2u_3\in G'$, a contradiction.
\end{proof}

\begin{lemma}\label{lem: no induced anticycle greater than 5}
Let $G$ be a (diamond, gap)-free graph and let $s\geq 1$. For every minimal generator $m$ of $I(G)^s$, the graph associated to $(I(G)^{s+1}:m)^{\pol}$ does not contain induced $C_n^c$ for all $n\geq 6$.
\end{lemma}

\begin{proof} 
Let $G'$ be the gap-free graph which is associated to $(I(G)^{s+1}:m)^{\pol}$ as in Theorem~\ref{thm: associated graph to colon ideal}. Assume for a contradiction $G'$ has induced $C_n^c$ for some $n\geq 6$ where $C_n=(v_1\dots v_n)$. From Theorem~\ref{thm: associated graph to colon ideal}, the anticycle $C_n^c$ is an induced subgraph of $G$. Let $m=e_1\dots e_s$ for some edges $e_1,\dots , e_s$ of $G$ and let $e_1=ab$. From Lemma~\ref{lem: anticycle does not intersect s fold product} it follows that neither $a$ nor $b$ belongs to $C_n$. Since there is no gap between $v_1v_3$ and $ab$, without loss of generality assume that $v_1a\in G$. Notice that $v_2b\notin G$ since otherwise $v_1abv_2$ is an even-connection and $v_1v_2\in G'$, which is a contradiction. Similarly, $v_nb\notin G$ as otherwise $v_1abv_n$ would be even-connection and that would require $v_1v_n\in G'$.  Since there is no gap between $v_2v_n$ and $ab$, without loss of generality assume that $v_2a\in G$. Observe that $v_3b\notin G$ as otherwise $v_2abv_3$ would be even connection and that would require $v_2v_3\in G'$.  Because there is no gap between $v_3v_n$ and $ab$, either $v_na\in G$ or $v_3a\in G$. Without loss of generality, suppose $v_3a\in G$. If $av_5\in G$, then the induced subgraph of $G$ on the vertices $v_2, v_3, v_5, a$ is a diamond, which is a contradiction. Otherwise, the induced subgraph of $G$ on the vertices $v_1, v_3, v_5, a$ is a diamond, again a contradiction.  
\end{proof}

\begin{lemma}\label{lem:colon ideals with regularity 2} Let $G$ be a gap-free and diamond-free graph with $\omega(G)=3$. Let $K_3=(abc)$ be a dominating clique of $G$.
	\begin{enumerate}[$(i)$]
		\item For all $s\geq 1$ and for all $(s-1)$-fold product $e_1\dots e_{s-1}$ edges of $G$, every induced $C_5$ in the graph of $(I(G)^{s+1}:(ab)e_1\dots e_{s-1})^{\pol}$ contains at least $2$ vertices of $K_3$.
		\item For all $s\geq 1$ and for all $(s-1)$-fold product $e_1\dots e_{s-1}$ edges of $G$, $\reg((I(G)^{s+1}:(ab)e_1\dots e_{s-1})^{\pol})=2$.
	\end{enumerate} 
\end{lemma}

\begin{proof}
	$(i)$  Let $C_5=(u_1u_2u_3u_4u_5)$ be an induced cycle of $G'$ where $G'$ is the graph associated to $(I(G)^{s+1}:(ab)e_1\dots e_{s-1})^{\pol}$ as in Theorem~\ref{thm: associated graph to colon ideal}. Then since $G$ is a subgraph of $G'$ and $V(C_5)\subseteq V(G)$, the cycle $C_5$ is also an induced subgraph of $G$. Keeping Lemma~\ref{lem: diamond free structure} in mind, assume for a contradiction $V(C_5)\cap V(K_3)$ has at most one vertex.
	
	\emph{Case 1}: Suppose that $V(C_5)\cap\{a,b,c\}=\emptyset$. Since no set of $3$ vertices of $C_5$ is independent, each of $N(a)\setminus V(K_3), N(b)\setminus V(K_3)$ and $N(c)\setminus V(K_3)$ contains at least one vertex of $C_5$. Without loss of generality assume that $(N(a)\setminus V(K_3))\cap V(C_5)=\{u_5\}, (N(b)\setminus V(K_3))\cap V(C_5)=\{u_1,u_3\}$ and $(N(c)\setminus V(K_3))\cap V(C_5)=\{u_2,u_4\}$. Then, there is no edge in $G$ that connects $u_2u_3$ and $au_5$, which is absurd because $G$ is gap-free.
	
	\emph{Case 2}: Suppose that $|V(C_5)\cap \{a,b,c\}|=1$. Without loss of generality assume that $a=u_1\in V(C_5)\cap \{a,b,c\}$. Then $(N(a)\setminus V(K_3))\cap V(C_5)=\{u_2, u_5\}$. Without loss of generality we may assume that $u_4$ is adjacent to $b$ and $u_3$ is adjacent to $c$ in $G$. But then $u_2abu_4$ is an even-connection and $u_2u_4\in G'$. This is a contradiction because $C_5$ is an induced cycle of $G'$.
	
   $(ii)$ Let $G'$ be the gap-free graph which is associated to $(I(G)^{s+1}:(ab)e_1\dots e_{s-1})^{\pol}$ as in Theorem~\ref{thm: associated graph to colon ideal}. From Fr\"{o}berg's Theorem and Lemma~\ref{lem: no induced anticycle greater than 5} it suffices to show that $G'$ has no induced $C_5$. Assume for a contradiction $G'$ has induced cycle $C_5=(u_1\dots u_5)$. From Lemma~\ref{lem: anticycle does not intersect s fold product} we have $\{a,b\}\cap \{u_1,\dots ,u_5 \}=\emptyset $. On the other hand, by part $(i)$ we have $|\{a,b,c\}\cap \{u_1,\dots , u_5\}|\geq 2$ which is a contradiction.
 \end{proof}
\section{Regularity of powers of (diamond, gap)-free graphs}
The authors of \cite{arbib} classified all imperfect $P_5$-free and diamond-free graphs based on a special family of graphs. Our main result Theorem~\ref{thm: regularity of powers of gap and diamond free graphs} will be based on this classification. To state Arbib and Mosca's theorem, we need some notation first. All the graphs in \cite[Fig.~3]{arbib} contain an induced $C_5=(u_0u_1u_2u_3u_4)$. Let $D_k$ denote the set of points at distance $k$ from this $C_5$.
\begin{theorem}\cite[Theorem~1.6]{arbib} Any connected ($P_5$, diamond)-free graph that properly contains an induced $C_5$ can be obtained from $C_5$ or from a graph among those of \cite[Fig.~3]{arbib} by
multiplying some $v\in D_1\cup C_5$ and/or substituting a $P_3$-free graph for some $v\in D_2$.
\end{theorem}

Since any gap-free graph is $P_5$-free we obtain the following corollary.
 \begin{corollary}\label{thm:list of graphs}
 	Any connected (gap, diamond)-free graph that properly contains an induced $C_5$ is either one of the graphs in Fig.~\ref{fig:classification of graphs} or it can be obtained from $C_5$ or from a graph among those of Fig.~\ref{fig:classification of graphs} by multiplying some vertices which do not belong to any triangles.
 \end{corollary}
 \begin{proof}
First note that multiplying a vertex of a triangle by $k>1$ yields a diamond. Observe that $G_{10}$ and $G_0$ are the only graphs in \cite[Fig.~3]{arbib} that contain vertices of distance $2$ from the induced $C_5=(u_0u_1u_2u_3u_4)$. Observe that in Fig.~\ref{fig:additional graph} there is a gap between the edges $u_0a_0$ and $b_0b_2$ of $G_4$. Therefore a (gap, diamond)-free graph cannot be obtained from $G_4$ by Proposition~\ref{prop: multiplying vertices}. Let $H$ be a graph that contains at least one edge. We claim that substituting $H$ for $y$ in $G_{10}$ or $G_0$ yields a gap. Indeed, if $y_1y_2$ is an edge of $H$, then there is a gap between $y_1y_2$ and $u_1u_2$.
 \end{proof}
 
 \begin{remark}\label{rk: vertex multiplication effect on cycle}
 Let $H$ be a graph and let $u$ be a vertex of $H$. Let $G$ be obtained from $H$ by multiplying the vertex $u$ by the independent set $U$. If $C_n$ is an induced cycle of $G$ where $n\geq 5$, then $C_n$ contains at most one vertex from $U$. In particular, if $H$ is $C_n$-free where $n\geq 5$, then so is $G$.
 \end{remark}
 The proof of the next lemma is computer aided. We used Maple to list the induced $5$-cycles of graphs in Fig.~\ref{fig:classification of graphs}.
  \begin{lemma}\label{lem: computer aided lemma}
 		Let $G$ be a (gap, diamond)-free graph which is obtained from a graph among the graphs $G_1, G_2,\dots , G_9$ in Fig.~\ref{fig:classification of graphs} by vertex multiplication. Let $C_5$ be an induced subgraph of $G$ and $e=\{a,b\}$ be an edge of $G$ such that $e\cap V(C_5)=\emptyset$. Then at least one of the following statements holds.
 		\begin{enumerate}[$(i)$]
 			\item $e\in E(K_3)$ for some dominating clique $K_3$ of $G$.
 			\item There exists distinct $u,v\in V(C_5)$ such that $au\in G, bv\in G$ and $uv\notin C_5$.
 		\end{enumerate}
 \end{lemma}
 \begin{proof}
 	 Using Corollary~\ref{thm:list of graphs} and keeping Remark~\ref{rk: vertex multiplication effect on cycle} in mind, we consider cases.
 	 
 	 \emph{Case 1}: Let $G$ be obtained from $G_1$ by replacing the vertices $u_0, u_1, u_4$ respectively with the independent sets $U_0, U_1, U_4$. Since $(u_0\dots u_4)$ is the only induced $5$-cycle of $G_1$, we may assume that $C_5=(u_0^1 u_1^1 u_2u_3u_4^1)$ for some $u_i^1\in U_i$. Since $e$ does not intersect $C_5$, by the symmetry of the graph we may assume $e=u_0^2 u_4^2$ or $e=u_0^2a_0$ for some $u_i^2\in U_i\setminus\{u_i^1\}$. In both cases, one vertex of $e$ is adjacent to $u_3$ and the other one is adjacent to $u_1^1$. Since $u_1^1u_3\notin C_5$, the statement $(ii)$ holds.
 	
 	 \emph{Case 2}: Suppose that $G$ is obtained from $G_2$ by replacing the vertex $u_1$ with $\{u_1^1,\dots ,u_1^k\}$. Since $(u_0\dots u_4)$ is the only induced $5$-cycle of $G_2$, we may assume that $C_5=(u_0u_1^1u_2u_3u_4)$. But then every edge of $G$ intersects $C_5$.
 	 
 	 \emph{Case 3}: Suppose that $G$ is obtained from $G_3$ by multiplying $u_0, u_2, u_3, u_4$. For each $i=0,2,3,4$ let $u_i$ be replaced with $\{u_i^1,\dots ,u_i^{k_i}\}$ in $G$. Suppose $e$ does not belong to the dominating clique $(b_0b_2u_1)$. Since the induced $5$-cycles of $G_3$ are $(u_0\dots u_4), (b_0u_1u_2u_3u_4)$ and $(b_2u_1u_0u_4u_3)$, the symmetry of the graph allows us to consider the following cases. 
 	  
 	 \emph{Case 3.1}: Suppose $C_5=(b_0u_1u_2^1u_3^1u_4^1)$. Then $e$ can have forms $u_0^iu_4^j, u_3^iu_4^j, u_3^iu_2^j$ or $u_3^ib_2$. If $e=u_0^iu_4^j$, then $(ii)$ holds because the endpoints of $e$ are respectively adjacent to $u_1$ and $u_3^1$ and $u_1u_3^1\notin C_5$. Similarly, if $e= u_3^iu_4^j$, then the endpoints of $e$ are respectively adjacent to $u_2^1$ and $b_0$ but $u_2^1b_0\notin C_5$. Lastly, if $e=u_3^iu_2^j$ or $e=u_3^ib_2$ then the endpoints of $e$ are respectively adjacent to  $u_4^1$ and $u_1$ but $u_1u_4^1\notin C_5$.
 	 
 	 \emph{Case 3.2}: Suppose  $C_5=(u_0^1u_1u_2^1u_3^1 u_4^1)$. Notice that if $e=u_p^iu_q^j$ for some $u_pu_q\in C_5$, then $(ii)$ clearly holds. By the symmetry of the graph, we may assume that $e=b_0u_4^i$. Then $u_3^1u_4^i\in G$, $b_0u_1\in G$ but $u_1u_3^1\notin G$ and $(ii)$ is satisfied.

 	 \emph{Case 4}: Suppose that $G$ is obtained from $G_5$ by multiplying the vertices $u_1, u_4$. Observe that $K_3=(b_1b_4u_0)$ is a dominating clique and every induced $5$-cycle of $G$ contains the vertices $u_0, u_2, u_3$. It follows that  every edge of $G$ either belongs to $K_3$ or intersects $C_5$.
 	 
 	 \emph{Case 5}: Suppose that $G$ is obtained from $G_6$ by multiplying the vertices $u_1, u_4, b_2$. Suppose that $u_1$ and $b_2$ are respectively replaced with the independent sets $U$ and $B$. Note that $(b_1b_4u_0)$ and $(a_0u_2u_3)$ are dominating cliques of $G$. Suppose that $e$ does not belong to these dominating triangles. Note that every induced $5$-cycle of $G$ contains both $u_0$ and $u_3$. So, $e$ contains neither $u_0$ nor $u_3$. Then $e$ can take the forms $uu_2, ub, b_1u_2$ or $b_1b$ for some $u\in U, b\in B$. In each case, one vertex of $e$ is adjacent to $u_0$ and the other one is adjacent to $u_3$. But since $u_0u_3\notin C_5$, the statement $(ii)$ holds.
 	 
 	 \emph{Case 6}: Suppose that $G$ is obtained from $G_7$ by multiplying the vertices $u_1, b_3, b_4$. Suppose that $b_3$ and $b_4$ are respectively replaced with the independent sets $B_3$ and $B_4$. Note that $(a_2u_0u_4)$ and $(a_0u_2u_3)$ are dominating cliques of $G$. Suppose that $e$ does not belong to these dominating triangles. Note that every induced $5$-cycle of $G$ contains both $u_0$ and $u_2$. So, $e$ contains neither $u_0$ nor $u_2$. Then $e$ can take the forms $u_3u_4, u_3b_4', b_3'u_4$ or $b_3'b_4'$ for some $b_3'\in B_3, b_4'\in B_4$. In each case, one vertex of $e$ is adjacent to $u_0$ and the other one is adjacent to $u_2$. But since $u_0u_2\notin C_5$, the statement $(ii)$ holds.

 	 \emph{Case 7}: Suppose that $G$ is obtained from $G_8$ by multiplying the vertices $b_2, b_4, u_1$. Suppose that $u_1$ and $b_2$ are respectively replaced with the independent sets $U$ and $B$. Note that $(a_2u_0u_4)$ and $(a_0u_2u_3)$ are dominating cliques of $G$. Suppose that $e$ does not belong to these dominating triangles. Note that every induced $5$-cycle of $G$ contains both $u_0$ and $u_3$. So, $e$ contains neither $u_0$ nor $u_3$. Then $e$ can take the forms $ub, uu_2, a_2u_2$ or $a_2b$ for some $b\in B, u\in U$. In each case, one vertex of $e$ is adjacent to $u_0$ and the other one is adjacent to $u_3$. But since $u_0u_3\notin C_5$, the statement $(ii)$ holds.
 	 
 	 \emph{Case 8}: Since every vertex of $G_9$ belongs to a triangle, $G_9$ does not generate any other graphs. Let $e$ be an edge which does not belong to any of the dominating triangles $(a_0 u_2 u_3)$, $(u_0 a_2 u_4)$ and $(b_0 b_2 u_1)$. Examining all $5$-cycles of $G_9$, we see that there are $3$ cases to consider.
 	 
 	 \emph{Case 8.1}: Suppose $C_5$ contains $a_0, a_2$ and $b_0$. Then $e$ is equal to one of $u_3u_4, u_3b_2, u_2u_1$ or $u_0u_1$. If $e$ is equal to $u_3u_4$ or  $u_3b_2$, then one vertex of $e$ is adjacent to $a_0$ and the other one is adjacent to $a_2$. But $a_0a_2\notin C_5$ so $(ii)$ holds. Otherwise, one vertex of $e$ is adjacent to $a_2$ and the other one is adjacent to $b_0$ and $b_0a_2\notin C_5$.
 	 
 	 \emph{Case 8.2}: Suppose $C_5$ contains $b_2, u_0$ and $u_3$. Then $e$ is equal to one of $b_0u_4, b_0a_0, a_2u_2$ or $u_1u_2$.  If $e$ is equal to $b_0u_4$ or $b_0a_0$, then one vertex of $e$ is adjacent to $b_2$ and the other one is adjacent to $u_0$. But $b_2u_0\notin C_5$ so $(ii)$ holds. Otherwise, one vertex of $e$ is adjacent to $u_0$ and the other one is adjacent to $u_3$ and $u_0u_3\notin C_5$.

 	 \emph{Case 8.3}: Suppose $C_5$ contains $u_1, u_2$ and $u_4$. Then $e$ is equal to one of $b_0a_0, u_0a_0, b_2a_2$ or $b_2u_3$. If $e$ is equal to $b_0a_0$ or  $u_0a_0$, then one vertex of $e$ is adjacent to $u_4$ and the other one is adjacent to $u_2$. But $u_2u_4\notin C_5$ so $(ii)$ holds. Otherwise, one vertex of $e$ is adjacent to $u_1$ and the other one is adjacent to $u_4$. But $u_1u_4\notin C_5$, completing the proof. 
 \end{proof}

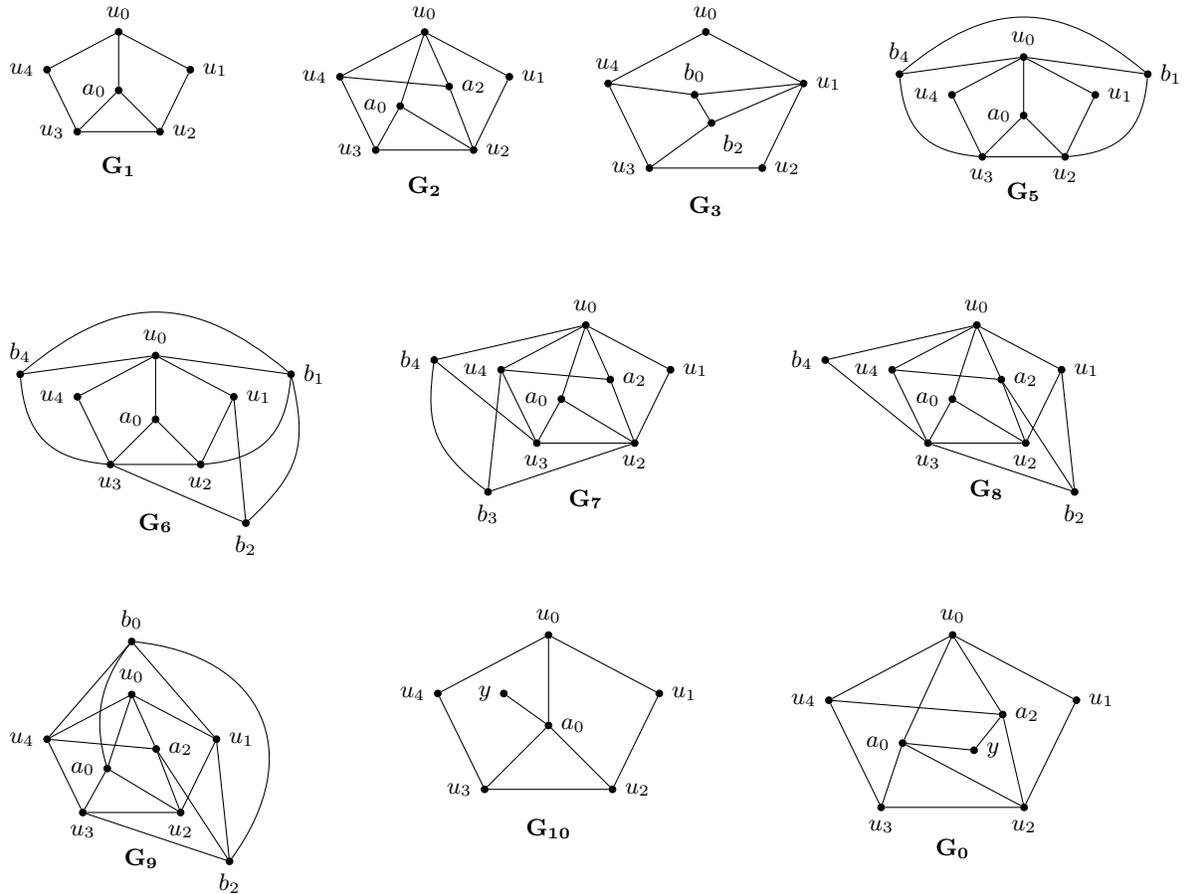
\begin{figure}[htp]

\begin{multicols}{4}
\footnotesize
\begin{tikzpicture}
[scale=0.55, vertices/.style={draw, fill=black, circle, minimum size = 4pt, inner sep=0.5pt}, another/.style={draw, fill=black, circle, minimum size = 2.5pt, inner sep=0.1pt}]
						\node[another, label=left:{$a_0$}] (a_0) at (0,0) {};
						\node[another, label=above:{$u_0$}] (u_0) at (0,1.4142) {};
						\node[another, label=right:{$u_1$}] (u_1) at (1.732, 0.5) {};%
						\node[another, label=right:{$u_2$}] (u_2) at (1,-1) {};
						\node[another, label=left:{$u_3$}] (u_3) at (-1,-1) {};
						\node[another, label=left:{$u_4$}] (u_4) at (-1.732, 0.5) {};%
						\node[label=below:{$\mathbf{G_1}$}](G_1) at (0,-1.2){};
						\foreach \to/\from in {a_0/u_0, a_0/u_2, a_0/u_3, u_0/u_1, u_1/u_2, u_2/u_3, u_3/u_4, u_4/u_0}
						\draw [-] (\to)--(\from);
						\end{tikzpicture}
\columnbreak 

\begin{tikzpicture}
[scale=0.65, vertices/.style={draw, fill=black, circle, minimum size = 4pt, inner sep=0.5pt}, another/.style={draw, fill=black, circle, minimum size = 2.5pt, inner sep=0.1pt}]
						\node[another, label=left:{$a_0$}] (a_0) at (-0.5,-0.1) {};
						\node[another, label=right:{$a_2$}] (a_2) at (0.5,0.3) {};
						\node[another, label=above:{$u_0$}] (u_0) at (0,1.4142) {};
						\node[another, label= right:{$u_1$}] (u_1) at (1.732, 0.5) {};%
						\node[another, label=right:{$u_2$}] (u_2) at (1,-1) {};
						\node[another, label=left:{$u_3$}] (u_3) at (-1,-1) {};
						\node[another, label=left:{$u_4$}] (u_4) at (-1.732, 0.5) {};%
						\node[label=below:{$\mathbf{G_2}$}](G_1) at (0,-1.2){};
						\foreach \to/\from in {a_0/u_0, a_0/u_2, a_0/u_3, a_2/u_0, a_2/u_4, a_2/u_2, u_0/u_1, u_1/u_2, u_2/u_3, u_3/u_4, u_4/u_0}
						\draw [-] (\to)--(\from);
						\end{tikzpicture}
\columnbreak

\begin{tikzpicture}
[scale=0.75, vertices/.style={draw, fill=black, circle, minimum size = 4pt, inner sep=0.5pt}, another/.style={draw, fill=black, circle, minimum size = 2.5pt, inner sep=0.1pt}]
						\node[another, label=above:{$b_0$}] (b_0) at (-0.2,0.3) {};
						\node[another, label=below right:{$b_2$}] (b_2) at (0.1,-0.2) {};
                                                  \node[another, label=above:{$u_0$}] (u_0) at (0,1.4142) {};
						\node[another, label=right:{$u_1$}] (u_1) at (1.732, 0.5) {};%
						\node[another, label=right:{$u_2$}] (u_2) at (1,-1) {};
						\node[another, label=left:{$u_3$}] (u_3) at (-1,-1) {};
						\node[another, label=above:{$u_4$}] (u_4) at (-1.732, 0.5) {};%
						\node[label=below:{$\mathbf{G_3}$}](G_1) at (0,-1.2){};
						\foreach \to/\from in {b_0/u_4, b_0/u_1, b_0/b_2, b_2/u_1, b_2/u_3, u_0/u_1, u_1/u_2, u_2/u_3, u_3/u_4, u_4/u_0}
						\draw [-] (\to)--(\from);
						\end{tikzpicture}
\columnbreak

\begin{tikzpicture}
[scale=0.55, vertices/.style={draw, fill=black, circle, minimum size = 4pt, inner sep=0.5pt}, another/.style={draw, fill=black, circle, minimum size = 2.5pt, inner sep=0.1pt}]
						\node[another, label=left:{$a_0$}] (a_0) at (0,0) {};
						\node[another, label=above:{$u_0$}] (u_0) at (0,1.4142) {};
						\node[another, label=right:{$u_1$}] (u_1) at (1.732, 0.5) {};%
						\node[another, label=below:{$u_2$}] (u_2) at (1,-1) {};
						\node[another, label=below:{$u_3$}] (u_3) at (-1,-1) {};
						\node[another, label=left:{$u_4$}] (u_4) at (-1.732, 0.5) {};%
						\node[another, label=right:{$b_1$}] (b_1) at (3, 1) {};%
					         \node[another, label=above:{$b_4$}] (b_4) at (-3, 1) {};%
					         \node[label=below:{$\mathbf{G_5}$}](G_1) at (0,-1.2){};
						\foreach \to/\from in {b_1/u_0, b_4/u_0, a_0/u_0, a_0/u_2, a_0/u_3, u_0/u_1, u_1/u_2, u_2/u_3, u_3/u_4, u_4/u_0}
						\draw [-] (\to)--(\from);
						
						 \path[every another/.style={font=\sffamily\small}]
						 (b_1)   edge [looseness=1.2, bend left=40] node [right] {} (u_2)
						  (b_4)   edge [looseness=1.2, bend right=40] node [right] {} (u_3)
                                                 (b_1)   edge [looseness=1.2, bend right=40] node [right] {} (b_4);
         
						\end{tikzpicture}
\end{multicols}

\begin{multicols}{3}
\footnotesize
\begin{tikzpicture}
[scale=0.60, vertices/.style={draw, fill=black, circle, minimum size = 4pt, inner sep=0.5pt}, another/.style={draw, fill=black, circle, minimum size = 2.5pt, inner sep=0.1pt}]
						\node[another, label=left:{$a_0$}] (a_0) at (0,0) {};
						\node[another, label=above:{$u_0$}] (u_0) at (0,1.4142) {};
						\node[another, label=right:{{$u_1$}}] (u_1) at (1.732, 0.5) {};%
						\node[another, label=below:{$u_2$}] (u_2) at (1,-1) {};
						\node[another, label=below:{$u_3$}] (u_3) at (-1,-1) {};
						\node[another, label=left:{$u_4$}] (u_4) at (-1.732, 0.5) {};%
						\node[another, label=right:{$b_1$}] (b_1) at (3, 1) {};%
					         \node[another, label=above:{$b_4$}] (b_4) at (-3, 1) {};%
					         \node[another, label=below:{$b_2$}] (b_2) at (2,-2.3) {};
					         \node[label=below:{$\mathbf{G_6}$}](G_6) at (0,-1.7){};
						\foreach \to/\from in {b_2/u_3, b_2/u_1, b_1/u_0, b_4/u_0, a_0/u_0, a_0/u_2, a_0/u_3, u_0/u_1, u_1/u_2, u_2/u_3, u_3/u_4, u_4/u_0}
						\draw [-] (\to)--(\from);
						
						 \path[every another/.style={font=\sffamily\small}]
						 (b_1)   edge [looseness=1.2, bend left=30] node [right] {} (b_2)

						 (b_1)   edge [looseness=1.2, bend left=40] node [right] {} (u_2)
						  (b_4)   edge [looseness=1.2, bend right=40] node [right] {} (u_3)
                                                 (b_1)   edge [looseness=1.2, bend right=40] node [right] {} (b_4);
         
						\end{tikzpicture}

\columnbreak 

\begin{tikzpicture}
[scale=0.65, vertices/.style={draw, fill=black, circle, minimum size = 4pt, inner sep=0.5pt}, another/.style={draw, fill=black, circle, minimum size = 2.5pt, inner sep=0.1pt}]
						\node[another, label=left:{$a_0$}] (a_0) at (-0.5,-0.1) {};
						\node[another, label=right:{$a_2$}] (a_2) at (0.5,0.3) {};
						\node[another, label=above:{$u_0$}] (u_0) at (0,1.4142) {};
						\node[another, label= right:{$u_1$}] (u_1) at (1.732, 0.5) {};%
						\node[another, label=below:{$u_2$}] (u_2) at (1,-1) {};
						\node[another, label=below:{$u_3$}] (u_3) at (-1,-1) {};
						\node[another, label=left:{$u_4$}] (u_4) at (-1.732, 0.5) {};%
						\node[another, label=below:{$b_3$}] (b_3) at (-2,-2) {};
						\node[another, label=left:{$b_4$}] (b_4) at (-3.1, 0.7) {};%
						\node[label=below:{$\mathbf{G_7}$}](G_7) at (0,-1.6){};

						\foreach \to/\from in {b_4/u_0, b_4/u_3, b_3/u_4, b_3/u_2, a_0/u_0, a_0/u_2, a_0/u_3, a_2/u_0, a_2/u_4, a_2/u_2, u_0/u_1, u_1/u_2, u_2/u_3, u_3/u_4, u_4/u_0}
						\draw [-] (\to)--(\from);
						
						 \path[every another/.style={font=\sffamily\small}]
						 (b_4)   edge [looseness=1.2, bend right=30] node [right] {} (b_3);
						\end{tikzpicture}
\columnbreak

\begin{tikzpicture}
[scale=0.65, vertices/.style={draw, fill=black, circle, minimum size = 4pt, inner sep=0.5pt}, another/.style={draw, fill=black, circle, minimum size = 2.5pt, inner sep=0.1pt}]
						\node[another, label=left:{$a_0$}] (a_0) at (-0.5,-0.1) {};
						\node[another, label=right:{$a_2$}] (a_2) at (0.5,0.3) {};
						\node[another, label=above:{$u_0$}] (u_0) at (0,1.4142) {};
						\node[another, label= right:{$u_1$}] (u_1) at (1.732, 0.5) {};%
						\node[another, label=below:{$u_2$}] (u_2) at (1,-1) {};
						\node[another, label=below:{$u_3$}] (u_3) at (-1,-1) {};
						\node[another, label=left:{$u_4$}] (u_4) at (-1.732, 0.5) {};%
						\node[another, label=below:{$b_2$}] (b_2) at (2,-2) {};
						\node[another, label=left:{$b_4$}] (b_4) at (-3.1, 0.7) {};%
						\node[label=below:{$\mathbf{G_8}$}](G_8) at (0.2,-1.4){};

						\foreach \to/\from in {b_4/u_0, b_4/u_3, b_2/a_2, b_2/u_1, b_2/u_3, a_0/u_0, a_0/u_2, a_0/u_3, a_2/u_0, a_2/u_4, a_2/u_2, u_0/u_1, u_1/u_2, u_2/u_3, u_3/u_4, u_4/u_0}
						\draw [-] (\to)--(\from);
						\end{tikzpicture}
\end{multicols}

\begin{multicols}{3}
\footnotesize
\begin{tikzpicture}
[scale=0.65, vertices/.style={draw, fill=black, circle, minimum size = 4pt, inner sep=0.5pt}, another/.style={draw, fill=black, circle, minimum size = 2.5pt, inner sep=0.1pt}]
						\node[another, label=left:{$a_0$}] (a_0) at (-0.5,-0.1) {};
						\node[another, label=right:{$a_2$}] (a_2) at (0.5,0.3) {};
						\node[another, label=above:{$u_0$}] (u_0) at (0,1.4142) {};
						\node[another, label= right:{$u_1$}] (u_1) at (1.732, 0.5) {};%
						\node[another, label=below:{$u_2$}] (u_2) at (1,-1) {};
						\node[another, label=below:{$u_3$}] (u_3) at (-1,-1) {};
						\node[another, label=left:{$u_4$}] (u_4) at (-1.732, 0.5) {};%
						\node[another, label=below:{$b_2$}] (b_2) at (2,-2) {};
						\node[another, label=above:{$b_0$}] (b_0) at (0, 2.5) {};%
						\node[label=below:{$\mathbf{G_9}$}](G_9) at (0.2,-1.4){};

						\foreach \to/\from in {b_0/u_4, b_0/u_1, b_2/a_2, b_2/u_1, b_2/u_3, a_0/u_0, a_0/u_2, a_0/u_3, a_2/u_0, a_2/u_4, a_2/u_2, u_0/u_1, u_1/u_2, u_2/u_3, u_3/u_4, u_4/u_0}
						\draw [-] (\to)--(\from);
						
						 \path[every another/.style={font=\sffamily\small}]
						 (b_0)   edge [looseness=1.0, bend right=27] node [right] {} (a_0)
                                                   (b_0)   edge [looseness=1.4, bend left=60] node [right] {} (b_2);

						\end{tikzpicture}
\columnbreak

\begin{tikzpicture}
[scale=0.85, vertices/.style={draw, fill=black, circle, minimum size = 4pt, inner sep=0.5pt}, another/.style={draw, fill=black, circle, minimum size = 2.5pt, inner sep=0.1pt}]
						\node[another, label=right:{$a_0$}] (a_0) at (0,0) {};
						\node[another, label=above:{$u_0$}] (u_0) at (0,1.4142) {};
						\node[another, label=right:{$u_1$}] (u_1) at (1.732, 0.5) {};%
						\node[another, label=right:{$u_2$}] (u_2) at (1,-1) {};
						\node[another, label=left:{$u_3$}] (u_3) at (-1,-1) {};
						\node[another, label=left:{$u_4$}] (u_4) at (-1.732, 0.5) {};%
						\node[another, label=left:{$y$}] (y) at (-0.7,0.5) {};
						\node[label=below:{$\mathbf{G_{10}}$}](G_1) at (0,-1.2){};

						\foreach \to/\from in {a_0/y, a_0/u_0, a_0/u_2, a_0/u_3, u_0/u_1, u_1/u_2, u_2/u_3, u_3/u_4, u_4/u_0}
						\draw [-] (\to)--(\from);
						\end{tikzpicture}
\columnbreak
\begin{tikzpicture}
[scale=0.95, vertices/.style={draw, fill=black, circle, minimum size = 4pt, inner sep=0.5pt}, another/.style={draw, fill=black, circle, minimum size = 2.5pt, inner sep=0.1pt}]
						\node[another, label=left:{$a_0$}] (a_0) at (-0.7,-0.1) {};
						\node[another, label=right:{$a_2$}] (a_2) at (0.7,0.3) {};
						\node[another, label=above:{$u_0$}] (u_0) at (0,1.4142) {};
						\node[another, label= right:{$u_1$}] (u_1) at (1.732, 0.5) {};%
						\node[another, label=below:{$u_2$}] (u_2) at (1,-1) {};
						\node[another, label=below:{$u_3$}] (u_3) at (-1,-1) {};
						\node[another, label=left:{$u_4$}] (u_4) at (-1.732, 0.5) {};%
						\node[another, label=right:{$y$}] (y) at (0.3, -0.2) {};
						\node[label=below:{$\mathbf{G_0}$}](G_0) at (0,-1.2){};

						\foreach \to/\from in {a_0/y, y/a_2, a_0/u_0, a_0/u_2, a_0/u_3, a_2/u_0, a_2/u_4, a_2/u_2, u_0/u_1, u_1/u_2, u_2/u_3, u_3/u_4, u_4/u_0}
						\draw [-] (\to)--(\from);
						\end{tikzpicture}
\end{multicols}
\caption{The list of graphs from which every (gap, diamond)-free graph in Corollary~\ref{thm:list of graphs} can be generated by vertex multiplication.}
\label{fig:classification of graphs}
\end{figure}

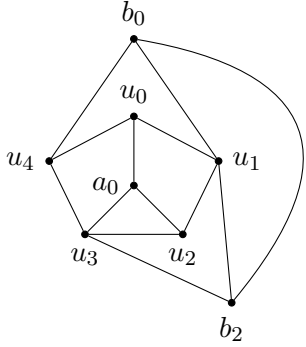
\begin{figure}[htp]
	\centering
	\begin{tikzpicture}
	[scale=0.65, vertices/.style={draw, fill=black, circle, minimum size = 4pt, inner sep=0.5pt}, another/.style={draw, fill=black, circle, minimum size = 2.5pt, inner sep=0.1pt}]
	\node[another, label=left:{$a_0$}] (a_0) at (0,0) {};
	\node[another, label=above:{$u_0$}] (u_0) at (0,1.4142) {};
	\node[another, label= right:{$u_1$}] (u_1) at (1.732, 0.5) {};%
	\node[another, label=below:{$u_2$}] (u_2) at (1,-1) {};
	\node[another, label=below:{$u_3$}] (u_3) at (-1,-1) {};
	\node[another, label=left:{$u_4$}] (u_4) at (-1.732, 0.5) {};%
	\node[another, label=below:{$b_2$}] (b_2) at (2,-2.4) {};
	\node[another, label=above:{$b_0$}] (b_0) at (0, 3) {};%
	
	\foreach \to/\from in {b_0/u_4, b_0/u_1, b_2/u_1, b_2/u_3, a_0/u_0, a_0/u_2, a_0/u_3, u_0/u_1, u_1/u_2, u_2/u_3, u_3/u_4, u_4/u_0}
	\draw [-] (\to)--(\from);
	
	\path[every another/.style={font=\sffamily\small}]
	(b_0)   edge [looseness=1.7, bend left=60] node [right] {} (b_2);
	\end{tikzpicture}
\caption{The graph $\mathbf{G_4}$ in \cite[Fig.~3]{arbib} which is not gap-free.}
\label{fig:additional graph}

\end{figure}

 \begin{theorem}\label{thm:family 1-9}
 	Let $G$ be a (gap, diamond)-free graph which is obtained from a graph among the graphs $G_1, G_2,\dots , G_9$ in Fig.\ref{fig:classification of graphs} by vertex multiplication. Then $\reg(I(G)^{s+1})=2s+2$ for all $s\geq 1$.
 \end{theorem}
 
 \begin{proof} Since each of the graphs $G_1, G_2, \dots ,G_9$ contains an induced $C_5$ and a triangle, all of them  have clique number $3$ by Lemma~\ref{lem: diamond free structure}(i). Let $s$ be fixed and let $m=e_1\dots e_s$ be a monomial generator of $I(G)^s$ for some edges $e_1:=ab, e_2,\dots ,e_s$ of $G$. We claim that $\reg(I(G)^{s+1}:m)\leq 2$ and then proof follows from Theorem~\ref{thm: regularity of gap and diamond free} and Theorem~\ref{thm: banerjee main bound}. Note that if $e_1$ is an edge of a dominating $K_3$ of $G$, then the proof follows from Lemma~\ref{lem:colon ideals with regularity 2}(ii). So, let us assume that $e_1$ does not belong to a dominating $K_3$. Let $G'$ be the gap-free graph which is associated to $(I(G)^{s+1}:m)^{\pol}$ as in Theorem~\ref{thm: associated graph to colon ideal}. From Fr\"oberg's theorem and Lemma~\ref{lem: no induced anticycle greater than 5}, it suffices to show that $G'$ has no induced cycle on $5$ vertices. Assume for a contradiction $G'$ has an induced $C_5$. Then $C_5$ is an induced subgraph of $G$ by Theorem~\ref{thm: associated graph to colon ideal}. We know that $e_1\cap V(C_5)=\emptyset$ from Lemma~\ref{lem: anticycle does not intersect s fold product}. Now, from Lemma~\ref{lem: computer aided lemma} there exists distinct  $u,v\in V(C_5)$ such that $au\in G, bv\in G$ and $uv\notin C_5$. Then $uabv$ is an even connection with respect to $e_1\dots e_s$ and 
 $uv\in G'$. But then since $C_5$ is an induced subgraph of $G'$ we get $uv\in C_5$, which is a contradiction completing the proof.
 \end{proof}
 	
 The proof of the theorem above does not work for (gap,diamond)-free graphs which are obtained from $G_0$ or $G_{10}$ in Fig.~\ref{fig:classification of graphs}. Using Macaulay 2, one can see $\reg(I(G_0)^2:ya_2)=3$ and $\reg(I(G_{10})^2:a_0y)=3$. 

\begin{proposition}\label{prop:10-family}
If $G$ is a (gap,diamond)-free graph with $I=I(G)$ which can be generated from $G_{10}$ in Fig.~\ref{fig:classification of graphs}, then $\reg(I^{s+1})=2s+2$ for all $s\geq 1$.
\end{proposition}

\begin{proof}
First note that from Corollary~\ref{thm:list of graphs} the graph $G$ is obtained from $G_{10}$ by multiplying some of the vertices $y, u_0, u_1, u_4$. Let $s\geq 1$ be fixed and let vertices $y, u_0, u_1, u_4$ of $G_{10}$ be respectively replaced by the independent sets $\{y_1,\dots ,y_k\}$, $\{u_0^1,\dots , u_0^{m_0}\}$, $\{u_1^1,\dots , u_1^{m_1}\}$ and $\{u_4^1,\dots , u_4^{m_4}\}$ in $G$. We first show that if $e_1\dots e_s$ is an $s$-fold product of edges of $G$ with the property that there exists $1\leq i \leq s$ such that for all $1\leq j\leq k$, $y_j\nmid e_i$, then the ideal $(I^{s+1}:e_1\dots e_s)$ has regularity $2$. To this end, suppose that $y_j\nmid e_1$ for all $j$.

From Theorems~\ref{thm: polarization}, \ref{thm: frobergs theorem}, \ref{thm: associated graph to colon ideal} and Lemma~\ref{lem: no induced anticycle greater than 5}, it suffices to show that the graph $G'$ which is associated to $(I^{s+1}:e_1\dots e_s)^{\pol}$ is $C_5$-free. Assume for a contradiction $G'$ has induced $C_5$. Then from Theorem~\ref{thm: associated graph to colon ideal}, $C_5$ is also an induced subgraph of $G$. Recalling Remark~\ref{rk: vertex multiplication effect on cycle}, observe that $C_5$ must contain both $u_2$ and $u_3$. Therefore, without loss of generality, we may assume that $C_5=(u_0^1u_1^1u_2u_3u_4^1)$. 
From Lemma~\ref{lem: anticycle does not intersect s fold product} and the symmetry of the graph, we may assume that $e_1=u_0^2u_4^2$ or $e_1=u_0^2a_0$. If $e_1=u_0^2u_4^2$, then $u_3u_4^2u_0^2u_1^1$ is an even-connection between $u_3$ and $u_1^1$ with respect to $e_1\dots e_s$ and $u_3u_1^1\in G'$, which is a contradiction because $C_5$ is an induced subgraph of $G'$. Similarly, if $e_1=u_0^2a_0$, then  $u_1^1u_0^2a_0u_3$ is an even-connection between $u_1^1$ and $u_3$ with respect to $e_1\dots e_s$ and $u_1^1u_3\in G'$, a contradiction. This completes the proof of our claim.

 Consider the order $e_1>\cdots > e_r>e_{r+1}>\cdots > e_{r+k}$ on the edges of $G$ where $e_{r+j}=a_0y_j$ for all $j=1,\dots ,k$. Let $M_1>\cdots >M_z>M_{z+1}>\cdots >M_{z+t}$ be the order on the minimal monomial generators of $I^s$ induced by the order on the edges as in Definition~\ref{def:order on the generators of the powers of ideal}. Let $z$ be the largest index such that $M_z$ has an expression $M_z=e_{i_1}\dots e_{i_s}$ such that $y_j\nmid e_{i_1}$ for all $1\leq j \leq k$. Observe that $M_{z+1},\dots ,M_{z+t}$ all have unique $s$-fold product expressions up to the permutation of the edges in the product. It follows from the previously proved claim that
\begin{equation}
\reg(I^{s+1}:M_i)=2 \text{ for all } i=1,\dots ,z. 
\end{equation}
Then repeated use of Lemma~\ref{lem:regularity after adding variables}, Theorem~\ref{thm:ordered colon ideals} and Theorem~\ref{thm:regularity bound by colon ideals} yields
\begin{equation*}
\begin{split}
\reg(I^{s+1}) & \leq \max\{\reg((I^{s+1}:M_1))+2s, \reg(I^{s+1}, M_1)\}\\
& \leq  \max\{2+2s, \reg((I^{s+1}, M_1):M_2)+2s, \reg((I^{s+1}, M_1, M_2))\}    \\
& =  \max\{2+2s, \reg(((I^{s+1}:M_2), \text{some variables}))+2s, \reg((I^{s+1}, M_1, M_2))\}   \\
& =  \max\{2+2s, \reg((I^{s+1}, M_1, M_2))\}   \\
& \quad \vdots \\
& \leq \max\{2+2s, \reg(I^{s+1},M_1,\dots ,M_z)\}.
\end{split}
\end{equation*} 
In order to repeat this process, it remains to show that for all $1\leq i \leq t$, 
\begin{equation}\label{eq:repeat process G_10 family}\reg((I^{s+1}, M_1, \dots , M_z, \dots ,M_{z+i-1}):M_{z+i})\leq 2.
\end{equation}
To this end, let $1\leq i \leq t$ be fixed. Then $M_{z+i}=(y_{j_1}a_0)^{\alpha_1}(y_{j_2}a_0)^{\alpha_2}\dots (y_{j_q}a_0)^{\alpha_q}$ for some $\alpha_1,\dots , \alpha_q>0$ and $k\geq j_1>j_2>\dots > j_q \geq 1$. Observe that $(I^{s+1}:M_{z+i})=I$ and using Theorem~\ref{thm:ordered colon ideals} we obtain
$$((I^{s+1}, M_1, \dots , M_z, \dots ,M_{z+i-1}):M_{z+i})=(I, \text{some variables}).  $$
Notice that every variable $x\in\{u_0^1, u_0^2, \dots , u_0^{m_0}, u_2, u_3\}$ belongs to the variable generators of the ideal above. Indeed, for every such $x$, we have $$M_{z+i-p}=(xa_0)(y_{j_1}a_0)^{\alpha_1-1}(y_{j_2}a_0)^{\alpha_2}\dots (y_{j_q}a_0)^{\alpha_q}$$  for some $p>0$. Therefore we obtain
$$((I^{s+1}, M_1, \dots , M_z, \dots ,M_{z+i-1}):M_{z+i})=(I, \text{some variables}, u_0^1, u_0^2, \dots , u_0^{m_0}, u_2, u_3 ).  $$
Let $H$ be the graph with edge ideal $I(H)=(y_1a_0,\dots , y_ka_0)$. Then 
$$(I, \text{some variables}, u_0^1, u_0^2, \dots , u_0^{m_0}, u_2, u_3 )=(I(H), \text{some variables}). $$
Since $H^c$ is chordal, from Fr\"oberg's Theorem $\reg(I(H))=2$ and then Eq.~\eqref{eq:repeat process G_10 family} follows from Lemma~\ref{lem:regularity after adding variables}. Now we have
$$\reg(I^{s+1})\leq \max\{2+2s, \reg(I^{s+1}, I^s)\}=\max\{2+2s, \reg(I^s)\}.  $$
Hence the result follows by induction and Theorem~\ref{thm: regularity of gap and diamond free}.
\end{proof}
Since the monomial ideals are polarized by their generators, the following follows from the construction.
\begin{remark}\label{rk:polarization remark}
	If $I\subseteq S$ is a monomial ideal and $J\subseteq S$ is a square-free monomial ideal, then $(I+J)^{\pol}=I^{\pol}+J$.
\end{remark}

\begin{proposition}\label{prop:0-family}
	If $G$ is a (gap,diamond)-free graph with $I=I(G)$ which can be generated from $G_{0}$ in Fig.~\ref{fig:classification of graphs}, then $\reg(I^{s+1})=2s+2$ for all $s\geq 1$.
\end{proposition}
\begin{proof}
	First note that from Corollary~\ref{thm:list of graphs} the graph $G$ is obtained from $G_{0}$ by multiplying some of the vertices $u_1, y$. Let $s\geq 1$ be fixed and suppose that the vertices $u_1$ and $y$ of $G_0$ are respectively replaced by the independent sets $U=\{u_1^1,\dots , u_1^m\}$ and $Y=\{y_1,\dots ,y_k\}$ in $G$. We claim that if $e_1\dots e_s$ is an $s$-fold product of edges of $G$ with the property that there exists $1\leq i \leq s$ such that $x\nmid e_i$ for all $x\in U\cup Y$, then $\reg(I^{s+1}:e_1\dots e_s)=2$.
	
To this end, suppose that $x\nmid e_1$ for all $x\in U\cup Y$. From Theorems~\ref{thm: polarization}, \ref{thm: frobergs theorem}, \ref{thm: associated graph to colon ideal} and Lemma~\ref{lem: no induced anticycle greater than 5}, it suffices to show that the graph $G'$ which is associated to $(I^{s+1}:e_1\dots e_s)^{\pol}$ is $C_5$-free. Assume for a contradiction $G'$ has induced $C_5$. Then $C_5$ is also an induced subgraph of $G$ because of Theorem~\ref{thm: associated graph to colon ideal}. Observe that
	$(ya_2u_4u_3a_0)$ and $(u_0u_1u_2u_3u_4)$ are the only induced $5$-cycles of $G_0$. Because of Remark~\ref{rk: vertex multiplication effect on cycle} we may assume without loss of generality that $C_5=(y_1 a_2 u_4 u_3 a_0)$ or $C_5=(u_0 u_1^1 u_2 u_3 u_4)$. If $C_5=(y_1 a_2 u_4 u_3 a_0)$, then by Lemma~\ref{lem: anticycle does not intersect s fold product}, $e_1$ does not intersect any of $a_2, u_4, u_3, a_0$ and therefore $e_1$ contains a vertex from $U$, which is a contradiction. Similarly, if $C_5=(u_0 u_1^1 u_2 u_3 u_4)$, then $e_1$ does not intersect any of $u_0, u_2, u_3, u_4$ and therefore $e_1$ contains a vertex from $Y$, which is a contradiction. This completes the proof of our claim.

Let $e_1,\dots ,e_r$ be the edges of $G$ which do not contain any vertex from $U\cup Y$. Consider the order
\begin{equation*}
\begin{split}
e_1>e_2>\cdots >e_r > & y_1a_0>y_1a_2> y_2a_0> y_2a_2> \cdots >y_ka_0> y_ka_2 \\
&  > u_1^1u_0> u_1^1u_2>u_1^2u_0> u_1^2u_2> \cdots > u_1^mu_0> u_1^mu_2    
\end{split}
\end{equation*} 
on the edges of $G$. Consider the order $M_1>\cdots >M_z>M_{z+1}>\cdots >M_{z+t}$ on the minimal monomial generators of $I^s$ induced by the order on the edges as in Definition~\ref{def:order on the generators of the powers of ideal}. Let $z$ be the largest index such that $M_z$ has an expression $M_z=e_{i_1}\dots e_{i_s}$ such that $x\nmid e_{i_1}$ for all $x\in Y\cup U$. 
 It follows from the previously proved claim that
\begin{equation}
\reg(I^{s+1}:M_i)=2 \text{ for all } i=1,\dots ,z. 
\end{equation}
Then repeated use of Lemma~\ref{lem:regularity after adding variables}, Theorem~\ref{thm:ordered colon ideals} and Theorem~\ref{thm:regularity bound by colon ideals} yields
\begin{equation*}
\begin{split}
\reg(I^{s+1}) & \leq \max\{\reg((I^{s+1}:M_1))+2s, \reg((I^{s+1},M_1))\}\\
& \leq  \max\{2+2s, \reg((I^{s+1}, M_1):M_2)+2s, \reg((I^{s+1}, M_1, M_2))\}    \\
& =  \max\{2+2s, \reg(((I^{s+1}:M_2), \text{some variables}))+2s, \reg((I^{s+1}, M_1, M_2))\}   \\
& =  \max\{2+2s, \reg((I^{s+1}, M_1, M_2))\}   \\
& \quad \vdots \\
& \leq \max\{2+2s, \reg(I^{s+1},M_1,\dots ,M_z)\}.
\end{split}
\end{equation*} 
In order to repeat this process, it remains to show that for all $1\leq i \leq t$, 
\begin{equation}\label{eq: repeat process G_0 family}
\reg(J_i):=\reg((I^{s+1}, M_1, \dots , M_z, \dots ,M_{z+i-1}):M_{z+i})\leq 2.
\end{equation}
Let $1\leq i\leq t$ be fixed. Then $M_{z+i}$ has a maximal expression where the exponent of $xw$ is non-zero for some $x\in U\cup Y$ and $xw\in G$. Let $N(w)\setminus (U\cup Y)=\{w_1,\dots, w_{\tau}\}$. Then for every $1\leq \kappa \leq \tau$ we have  $M_{z+i}w_{\kappa}/x >M_{z+i}$ since $ww_{\kappa}\in G$ and $w,w_{\kappa}\notin U\cup Y$. Therefore from Theorem~\ref{thm: polarization}, Theorem~\ref{thm:ordered colon ideals} and Remark~\ref{rk:polarization remark} it follows that
\begin{equation*}
\begin{split}
\reg(J_i) & = \reg(((I^{s+1}:M_{z+i}), w_1, \dots, w_{\tau}, \text{some variables}))\\
& =  \reg(((I^{s+1}:M_{z+i})^{\pol}, w_1, \dots, w_{\tau}, \text{some variables}))
\end{split}
\end{equation*} 

Let $G'$ be the gap-free graph which is associated to $(I^{s+1}:M_{z+i})^{\pol}$ as in Theorem~\ref{thm: associated graph to colon ideal}. Using Lemma~\ref{lem:regularity after adding variables} it suffices to show that $\reg((I(G'), w_1, \dots, w_{\tau}))\leq 2$. If $(I(G'), w_1, \dots, w_{\tau})$ is generated in degree $1$, we have nothing to show. So, let us assume that $((I(G'), w_1, \dots, w_{\tau})=I(H)$ where $H=G'- \{w_1,\dots ,w_{\tau}\}$ has at least one edge. By Fr\"oberg's theorem and Lemma~\ref{lem: no induced anticycle greater than 5}, it remains to show that $H$ is $C_5$-free. Assume for a contradiction $C_5$ is an induced cycle of $H$. Then $C_5$ is an induced cycle of $G'$ since $H$ is an induced subgraph of $G'$. From Theorem~\ref{thm: associated graph to colon ideal} it follows that $C_5$ is an induced cycle of $G$. Observe that every induced $5$-cycle of $G_0$ contains both $u_3$ and $u_4$. Therefore $C_5$ contains both $u_3$ and $u_4$ by Remark~\ref{rk: vertex multiplication effect on cycle}. If $w=a_0$ or $w=u_2$, then $u_3\notin V(H)$, and $u_3\notin V(C_5)$, a contradiction. Similarly, if $w=a_2$ or $w=u_0$, then $u_4\notin V(H)$, and $u_4\notin V(C_5)$, a contradiction.

Now we have
$$\reg(I^{s+1})\leq \max\{2+2s, \reg(I^{s+1}, I^s)\}=\max\{2+2s, \reg(I^s)\}.  $$
Hence the result follows by induction and Theorem~\ref{thm: regularity of gap and diamond free}.
\end{proof}

Finally, we prove the main result of this paper.

\begin{theorem}\label{thm: regularity of powers of gap and diamond free graphs}
	If $G$ is a (gap, diamond)-free graph, then $\reg(I(G)^s)=2s$ for all $s\geq 2$.
\end{theorem}
\begin{proof}
We may assume $G$ is a connected gap-free and diamond-free graph since removal of isolated vertices does not change the edge ideal. If $\omega(G)\geq 4$, then the proof follows from Theorem~\ref{thm: frobergs theorem}, Lemma~\ref{lem: diamond free structure} and \cite[Theorem~3.2]{herzog hibi zheng}. If $\omega(G)<3$, then $G$ is cricket-free and the result follows from \cite[Theorem~6.17]{banerjee}. Therefore let us assume that $\omega(G)=3$. Note that $C_6^c$ is cricket-free since $C_6^c$ has no vertex with $4$ neighbors. If $G$ is $C_5$-free, then the result follows from combining Theorem~\ref{thm: frobergs theorem}, Lemma~\ref{lem: gap free and chordal families}, \cite[Theorem~ 6.17]{banerjee} and \cite[Theorem~3.2]{herzog hibi zheng}. Therefore let us assume that $G$ contains an induced cycle on $5$ vertices. If $G$ is obtained from $C_5$ by multiplying vertices, then $\omega(G)<3$ by Proposition~\ref{prop: multiplying vertices}. Otherwise, the result follows from combining Corollary~\ref{thm:list of graphs}, Theorem~\ref{thm:family 1-9}, Proposition~\ref{prop:10-family} and Proposition~\ref{prop:0-family}.
\end{proof}

%
\end{document}